\documentclass[10pt,draft]{article}
\usepackage{amsmath,amscd}
\usepackage{amssymb,latexsym,amsthm}
\usepackage{color}
\usepackage[spanish,english]{babel}
\usepackage{amssymb}
\usepackage{color}
\usepackage[latin1]{inputenc}
\usepackage{lscape}
\usepackage{fancyhdr}
\usepackage{pb-diagram}
\usepackage{amsfonts}

\numberwithin{equation}{section}
\newtheorem{theorem}{Theorem}[section]

\newtheorem{definition}[theorem]{Definition}
\newtheorem{proposition}[theorem]{Proposition}
\newtheorem{lemma}[theorem]{Lemma}

\newtheorem{corollary}[theorem]{Corollary}

\theoremstyle{definition}
\newtheorem{example}[theorem]{Example}

\newtheorem{remark}[theorem]{Remark}

\makeatletter
\def\@roman#1{\romannumeral #1}
\makeatother

\title{\textbf{Matrix computations on projective modules \\ using noncommutative Gr\"obner bases}}
\author{Claudia Gallego\\
\texttt{cmgallegoj@unal.edu.co}
\\ Seminario de Álgebra Constructiva - SAC$^2$\\ Departamento de Matemáticas\\ Universidad Nacional de
Colombia, Sede Bogot\'a}
\date{}
\begin{document}
\maketitle
\begin{abstract}
\noindent Constructive proofs of fact that a stably free left $S$-module $M$ with ${\rm
rank}(M)\geq \text{\rm sr}(S)$ is free, where $\text{\rm sr}(S)$ denotes the stable rank of an
arbitrary ring $S$, were developed in \cite{Gallego3} (see also \cite{Gallego6} and
\cite{Quadrat}). Additionally, in such papers, are presented algorithmic proofs for calculating
projective dimension, and to check whether a left $S$-module $M$ is stably free. Given a left
$A$-module $M$, with $A$ a bijective skew $PBW$ extension, we will use these results and Gröbner
bases theory, to establish algorithms that allow us  to calculate effectively the projective
dimension for this module, to check whether is stably free, to construct minimal presentations, and
to obtain bases for free modules.
\bigskip

\bigskip

\end{abstract}
\section{Introduction}
\noindent The Gröbner bases theory provides us with a remarkable array of tools for the effective
calculation of diverse algebraic objects. We have developed this theory for skew $PBW$ extensions
(see \cite{Gallego6} and \cite{Gallego5}), which in turn has enabled us carry out calculations in a
broad class of noncommutative rings. In addition, given an arbitrary left $A$-module $M$, $A$ a
bijective skew $PBW$ extension, this Gröbner bases theory along with the existence of matrix
constructive results that allow to establish algorithms for calculating the projective dimension of
$M$, to check whether or not $M$ is stably free, and to obtain effectively a basis when $M$  is a
stably free module with ${\rm rank}(M)\geq \text{\rm sr}(S)$, will enable us to present throughout
current article, effective algorithms and computations of this kind, for modules defined on skew
$PBW$ extensions. The key tool for these algorithms will be the calculation of left and right
inverse matrices.

Related to the computation of projective dimension, we have that the following notable general
facts over arbitrary rings. $S$ will represent an arbitrary noncommutative ring.
\begin{theorem}\label{6.3.3}
Let $M$ be an $S$-module and
\begin{equation}\label{finiteprojectiveresolution}
0\to P_m\xrightarrow{f_{m}} P_{m-1}\xrightarrow{f_{m-1}}
P_{m-2}\xrightarrow{f_{m-2}} \cdots \xrightarrow{f_2}
P_1\xrightarrow{f_1} P_0\xrightarrow{f_0} M\longrightarrow 0
\end{equation}
a projective resolution of $M$. If $m\geq 2$ and there exists a
homomorphism $g_m:P_{m-1}\to P_m$ such that $g_mf_m=i_{P_m}$, then
we have the following projective resolution of $M$:
\begin{equation}
0\to P_{m-1}\xrightarrow{h_{m-1}} P_{m-2}\oplus P_m
\xrightarrow{h_{m-2}} P_{m-3}\xrightarrow{f_{m-3}} \cdots
\xrightarrow{f_2} P_1\xrightarrow{f_1} P_0\xrightarrow{f_0}
M\longrightarrow 0
\end{equation}
with
\begin{center}
$h_{m-1}:=\begin{bmatrix}f_{m-1}\\g_m\end{bmatrix}$, \ \
$h_{m-2}:=\begin{bmatrix}f_{m-2}& 0\end{bmatrix}$.
\end{center}
\end{theorem}
\begin{proof}
See \cite{Quadrat}, Proposition 20.
\end{proof}

\begin{corollary}\label{6.3.4}
Let $M$ be an $S$-module and
\begin{equation}\label{equ6.2.11}
0\to S^{s_m}\xrightarrow{f_{m}} S^{s_{m-1}}\xrightarrow{f_{m-1}}
S^{s_{m-2}}\xrightarrow{f_{m-2}} \cdots \xrightarrow{f_2}
S^{s_1}\xrightarrow{f_1} S^{s_0}\xrightarrow{f_0} M\longrightarrow
0
\end{equation}
a finite free resolution of $M$. Let $F_i$ be the matrix of $f_i$
in the canonical bases, $1\leq i\leq m$. Then,
\begin{enumerate}
\item[\rm (i)]If $m\geq 3$ and there exists a homomorphism $g_m:S^{s_{m-1}}\to
S^{s_m}$ such that $g_mf_m=i_{S^{s_m}}$, then we have the
following finite free resolution of $M$: {\footnotesize
\begin{equation}\label{equ6.3.5}
0\to S^{s_{m-1}}\xrightarrow{h_{m-1}} S^{s_{m-2}+s_m}
\xrightarrow{h_{m-2}} S^{s_{m-3}}\xrightarrow{f_{m-3}} \cdots
\xrightarrow{f_1} S^{s_0}\xrightarrow{f_0} M\longrightarrow 0
\end{equation}}
with
\begin{center}
$h_{m-1}:=\begin{bmatrix}f_{m-1}\\g_m\end{bmatrix}$, \ \
$h_{m-2}:=\begin{bmatrix}f_{m-2}& 0\end{bmatrix}$.
\end{center}
In a matrix notation, if $G_m$ is the matrix of $g_m$ and $H_j$ is
the matrix of $h_j$ in the canonical bases, $j=m-1,m-2$, then
\begin{center}
$H_{m-1}^T:=\begin{bmatrix}F_{m-1}^T& G_m^T\end{bmatrix}$, \ \
$H_{m-2}^T:=\begin{bmatrix}F_{m-2}^T\\0\end{bmatrix}$.
\end{center}
\item[\rm (ii)]If $m=2$ and there exists a homomorphism $g_2:S^{s_1}\to
S^{s_2}$ such that $g_2f_2=i_{S^{s_2}}$, then we have the
following finite presentation of $M$:
\begin{equation}\label{equ6.3.6}
0\to S^{s_1}\xrightarrow{h_1}S^{s_0+s_2}\xrightarrow{h_0}M\to 0,
\end{equation}
with
\begin{center}
$h_{1}:=\begin{bmatrix}f_{1}\\g_2\end{bmatrix}$, \ \
$h_{0}:=\begin{bmatrix}f_{0}& 0\end{bmatrix}$.
\end{center}
In a matrix notation,
\begin{center}
$H_{1}^T:=\begin{bmatrix}F_{1}^T& G_2^T\end{bmatrix}$, \ \
$H_{0}^T:=\begin{bmatrix}f_{0}\\0\end{bmatrix}$.
\end{center}
\end{enumerate}
\end{corollary}
\begin{proof}
See \cite{Quadrat}, Corollary 21.
\end{proof}

With respect to stably freeness, the following characterization holds.
\begin{theorem}\label{3.2.19a}
Let $M$ be an $S$-module with exact sequence $0\rightarrow
S^s\xrightarrow{f_1}S^r\xrightarrow{f_0}M\rightarrow 0$. Then,
$M^T\cong Ext_S^{1}(M,S)$ and the following conditions are
equivalent:
\begin{enumerate}
\item[\rm (i)]$M$ is stably free.
\item[\rm (ii)]$M$ is projective.
\item[\rm (iii)]$M^T=0$.
\item[\rm (iv)]$F_1^T$ has a right inverse.
\item[\rm (v)]$f_1$ has a left inverse.
\end{enumerate}
\end{theorem}
\begin{proof}
See \cite{Gallego6}, Theorem 2.2.15.
\end{proof}
Finally, regarding free modules, we include below a matrix constructive characterization.
\begin{lemma}\label{matricialhermite2}
Let $S$ be a ring and $M$ a stably free $S$-module given by a
minimal presentation
$S^s\xrightarrow{f_1}S^r\xrightarrow{f_0}M\rightarrow 0$. Let
$g_1:S^r\to S^s$ such that $g_1f_1=i_{S^s}$. Then the following
conditions are equivalent:
\begin{enumerate}
\item[\rm (i)]$M$ is free of dimension $r-s$.
\item[\rm (ii)]There exists a matrix $U\in GL_r(S)$ such that
$UG_1^T=\begin{bmatrix}I_s\\
0\end{bmatrix}$, where $G_1$ is the matrix of $g_1$ in the canonical bases. In such case, the last
$r-s$ columns of $U^T$ conform a basis for $M$. Moreover, the first $s$ columns of $U^T$ conform
the matrix $F_1$ of $f_1$ in the canonical bases.
\item[\rm (iii)]There exists
a matrix $V\in GL_r(S)$ such that $G_1^T$ coincides with the first
$s$ columns of $V$, i.e., $G_1^T$ can be completed to an
invertible matrix $V$ of $GL_r(S)$.
\end{enumerate}
\end{lemma}
\begin{proof}
See \cite{Gallego6}, Lemma 2.3.5.
\end{proof}
Some definitions, and elementary properties are necessary in what follows. These can also be
reviewed in \cite{Gallego3}. If $S$ is an arbitrary ring, $S^r$ denotes the left $S$-module of
columns of size $r\times 1$; for each  $S$-homomorphism $S^s\xrightarrow{f}S^r$, there is a matrix
associated to $f$ in the canonical bases of $S^r$ and $S^s$, denoted $F:=m(f)$, and disposed by
columns, i.e., $F\in M_{r\times s}(S)$; moreover, if $a\in S^s$, then
$f(\textbf{\emph{a}})=(\textbf{\emph{a}}^TF^T)^T$. It is straightforward to show that function
$m:Hom_{S}(S^s,S^r)\rightarrow M_{r\times s}(S)$ is bijective; and that if $S^r\xrightarrow{g}S^p$
is a homomorphism, then the matrix of $gf$ in the canonical bases is $m(gf)=(F^TG^T)^T$. Thus,
$f:S^r\rightarrow S^r$ is an isomorphism if and only if $F^T\in GL_r(S)$, and if $C\in M_r(S)$, we
have that columns of $C$ conform a basis of $S^r$ if and only if $C^T\in GL_r(S)$. When $S$ is
commutative, or when we consider right modules instead of left modules, we have that
$f(\textbf{\emph{a}})=F\textbf{\emph{a}}$ and, in such cases, the matrix of a compose homomorphism
$gf$ is given by $m(gf)=m(g)m(f)$. Further, $f:S^r\rightarrow S^r$ is an isomorphism if and only if
$F\in GL_r(S)$; besides, $C\in GL_r(S)$ if and only if its columns conform a basis of $S^r$ (see
section 1.1 in \cite{Gallego3}). Now, let $S$ be a ring; we say that $S$ satisfies the \textit{rank
condition} {\rm(}$\mathcal{RC}${\rm)} if for any integers $r,s\geq 1$, if $S^r\xrightarrow{f} S^s$
is an epimorphism, then $r\geq s$. Furthermore, $S$ is an $\mathcal{IBN}$ ring
{\rm(}\textit{invariant basis number}{\rm)} if for any integers $r,s\geq 1$, $S^r\cong S^s$ if and
only if $r=s$. It is well known that $\mathcal{RC}$ implies $\mathcal{IBN}$. From now on we will
assume that all rings considered in the present paper are $\mathcal{RC}$. We have the the following
elementary characterization for {\rm(}$\mathcal{RC}${\rm)} rings.
\begin{proposition}\label{2.1.4}
Let $S$ be a ring.
\begin{enumerate}
\item[\rm (i)]$S$ is $\mathcal{RC}$ if and only if given any matrix $F\in M_{s\times
r}(S)$ the following condition holds:
\begin{center}
if $F$ has a right inverse then $r\geq s$.
\end{center}
\item[\rm (ii)]$S$ is $\mathcal{RC}$ if and only if given any matrix $F\in M_{s\times
r}(S)$ the following condition holds:
\begin{center}
if $F$ has a left inverse then $s\geq r$.
\end{center}
\end{enumerate}
\end{proposition}
\begin{proof}
C.f. \cite{Gallego3}, Proposition 2.
\end{proof}
\section{Skew $PBW$ extensions}\label{definitionexamplesspbw}
\noindent In this section we recall the definition of skew $PBW$ (Poincaré-Birkhoff-Witt)
extensions defined firstly in \cite{Gallego2}, and we will review also some basic properties about
the polynomial interpretation of this kind of noncommutative rings. Two particular subclasses of
these extensions are recalled also.
\begin{definition}\label{gpbwextension}
Let $R$ and $A$ be rings. We say that $A$ is a \textit{skew $PBW$ extension of $R$} $($also called
a $\sigma-PBW$ extension of $R$$)$ if the following conditions hold:
\begin{enumerate}
\item[\rm (i)]$R\subseteq A$.
\item[\rm (ii)]There exist finite elements $x_1,\dots ,x_n\in A$ such $A$ is a left $R$-free module with basis
\begin{center}
${\rm Mon}(A):= \{x^{\alpha}=x_1^{\alpha_1}\cdots x_n^{\alpha_n}\mid \alpha=(\alpha_1,\dots
,\alpha_n)\in \mathbb{N}^n\}$.
\end{center}
\item[\rm (iii)]For every $1\leq i\leq n$ and $r\in R-\{0\}$ there exists $c_{i,r}\in R-\{0\}$ such that
\begin{equation}\label{sigmadefinicion1}
x_ir-c_{i,r}x_i\in R.
\end{equation}
\item[\rm (iv)]For every $1\leq i,j\leq n$ there exists $c_{i,j}\in R-\{0\}$ such that
\begin{equation}\label{sigmadefinicion2}
x_jx_i-c_{i,j}x_ix_j\in R+Rx_1+\cdots +Rx_n.
\end{equation}
Under these conditions we will write $A:=\sigma(R)\langle x_1,\dots ,x_n\rangle$.
\end{enumerate}
\end{definition}

A particular case of skew $PBW$ extension is when all derivations $\delta_i$ are zero. Another
interesting case is when all $\sigma_i$ are bijective and the constants $c_{ij}$ are invertible. We
recall the following definition (cf. \cite{Gallego2}).
\begin{definition}\label{sigmapbwderivationtype}
Let $A$ be a skew $PBW$ extension.
\begin{enumerate}
\item[\rm (a)]
$A$ is quasi-commutative if the conditions {\rm(}iii{\rm)} and {\rm(}iv{\rm)} in Definition
\ref{gpbwextension} are replaced by
\begin{enumerate}
\item[\rm (iii')]For every $1\leq i\leq n$ and $r\in R-\{0\}$ there exists $c_{i,r}\in R-\{0\}$ such that
\begin{equation}
x_ir=c_{i,r}x_i.
\end{equation}
\item[\rm (iv')]For every $1\leq i,j\leq n$ there exists $c_{i,j}\in R-\{0\}$ such that
\begin{equation}
x_jx_i=c_{i,j}x_ix_j.
\end{equation}
\end{enumerate}
\item[\rm (b)]$A$ is bijective if $\sigma_i$ is bijective for
every $1\leq i\leq n$ and $c_{i,j}$ is invertible for any $1\leq i<j\leq n$.
\end{enumerate}
\end{definition}
A remarkable property satisfies by skew $PBW$ extensions is presented below, which at the same time assures us that the algorithms used finish.
\begin{proposition}[Hilbert Basis Theorem]\label{1.3.4}
Let $A$ be a bijective skew $PBW$ extension of $R$. If $R$ is a left $($right$)$ Noetherian ring
then $A$ is also a left $($right$)$ Noetherian ring.
\end{proposition}
\begin{proof}
See \cite{lezamareyes1}, Corollary 2.4.
\end{proof}
Since the objects studied in the present paper are  skew $PBW$ extensions, it is necessary to
guarantee the $\mathcal{IBN}$ and $\mathcal{RC}$ properties for these rings. For this, we have the following important fact:
\begin{theorem}\label{618}
Let $A$ be a skew $PBW$ extension of a ring $R$. Then, $A$ is $\mathcal{RC}$ $(\mathcal{IBN})$ if
and only if $R$ is $\mathcal{RC}$ $(\mathcal{IBN})$.
\end{theorem}
\begin{proof}
See \cite{Gallego4}, Theorem 2.9.
\end{proof}
\begin{remark}
We developed the Gröbner bases theory for any bijective skew $PBW$ extension. \-Spe\-ci\-fi\-ca\-lly, we established a Buchberger's algorithm for these rings, the computation of syzygies module, as well as some applications as  membership problem, calculation of intersections, quotients, presentation of a module, computing free resolutions, the kernel and image of an homomorphism (see Chapter 5 and Chapter 6 in \cite{Gallego6}, or \cite{Gallego5}).
\end{remark}
\section{Computing the inverse of a matrix}
\noindent In this section we will present an algorithm that determines if a given rectangular
matrix over a bijective skew $PBW$ extension is left invertible, and in such case, the algorithm
computes one of its left inverses. A similar algorithm will be constructed for the right side case.
We start with the following elementary fact about left invertible matrices.
\begin{proposition}\label{LeftInverse}
Let $F$ be a rectangular matrix of size $r\times s$ with entries in a ring $S$. If $F$ has left
inverse, then $r\geq s$. Moreover, $F$ has a left inverse if and only if the module generated by
the rows of $F$ coincides with $S^s$.
\begin{proof}
The first statement follows from the fact that we are assuming the $S$ is $\mathcal{RC}$
(see Proposition \ref{2.1.4}). Now, suppose that $F$ has a left inverse $L\in
M_{s\times r}(S)$, i.e., $LF=I_s$. Define the following $S$-homomorphisms
\begin{equation*}
\begin{aligned}
f^{t}: S^r &\to S^s   \\
\textbf{\emph{a}} &\mapsto (\emph{\textbf{a}}^TF)^T
\end{aligned}
\qquad \qquad
\begin{aligned}
l^t:S^s &\to S^r\\
\textbf{\emph{b}}&\mapsto (\textbf{\emph{b}}^TL)^T,
\end{aligned}
\end{equation*}
then $m(f^t)=F^T$ and $m(l^t)=L^T$. Whence, $m(f^t
\circ l^t)=(LF)^T=I_{s}^{T}=I_s$, i.e, $f^t$ is an epimorphism. Hence, $Im(f^t)=S^s$, i.e., the
left submodule generated by the rows of $F$ coincides with the free module $S^s$. Conversely,
suppose that the module generated by the rows of $F$ coincides wit $S^s$, then for $f^t$ defined as
above, there exist $\textbf{\emph{a}}_{1}\ldots, \textbf{\emph{a}}_s\in S^r$ such that
$f^t(\textbf{\emph{a}}_i)=\textbf{\emph{e}}_{i}$ for each $1\leq i\leq s$, and where
$\textbf{\emph{e}}_{1},\ldots, \textbf{\emph{e}}_s$ denote the canonical vectors of $S^s$. Thus, if
$\textbf{\emph{a}}_{i}=\begin{bmatrix} a_{1i}& a_{2i} &\cdots& a_{ri} \end{bmatrix}^T$, we have
\[\textbf{\emph{a}}_{i}^TF=\begin{bmatrix} a_{1i}& a_{2i} &\cdots& a_{ri} \end{bmatrix}F=a_{1i}F_{(1)}+\cdots+ a_{ri}F_{(r)}=\textbf{\emph{e}}_i,\]
where $F_{(j)}$ denotes the $j$-th row of $F$, $1\leq j\leq r$. Therefore, if $L$ is the matrix
whose rows are the vectors $\textbf{\emph{a}}_{i}^T$, then $LF=I_s$, i.e., $F$ has a left inverse.
\end{proof}
\end{proposition}

\begin{corollary}\label{LeftInverseAlg}
Let $A$ be a bijective skew $PBW$ extension and let $F\in M_{r\times s}(A)$ be a rectangular matrix
over $A$. The algorithm below determines if $F$ is left invertible, and in the positive case, it
computes a left inverse of $F$:
\end{corollary}

\begin{center}
\fbox{\parbox[c]{12cm}{
\begin{center}
{\rm \textbf{Algorithm for the left inverse of a matrix}}
\end{center}
\begin{description}
\item[]{\rm \textbf{INPUT}:} A rectangular matrix $F\in M_{r\times s}(A)$
\item[]{\rm \textbf{OUTPUT}:} A matrix $L\in M_{s\times r}(A)$ satisfying $LF=I_s$ if it exists, and $0$ in other case
\item[]{\rm \textbf{INITIALIZATION}:}

{\rm \textbf{IF}} $r< s$

\quad {\rm \textbf{RETURN}} $0$

\quad \quad  {\rm \textbf{IF}} $r\geq s$, let
$G:=\{\textbf{\emph{g}}_1,\ldots,\textbf{\emph{g}}_t\}$ be a Gröbner basis for the left submodule
generated by rows of $F$ and let $\{\textbf{\emph{e}}_i\}_{i=1}^s$ be the canonical basis of $A^s$.
Use the division algorithm to verify if $\textbf{\emph{e}}_{i}\in \langle\,_AG\rangle$ for each $1\leq
i \leq s$.

{\rm \textbf{IF}} there exists some $\textbf{\emph{e}}_{i}$ such that $\textbf{\emph{e}}_i\notin
\langle G \rangle$,

\quad {\rm \textbf{RETURN}} $0$

{\rm \textbf{IF}} $\langle G\rangle=A^s$, let $H\in M_{r\times t}(A)$ with the property $G^T=H^TF$,
and consider $K:=[k_{ij}]\in M_{t\times s}$, where the $k_{ij}$'s are such that
$\textbf{\emph{e}}_{i}=k_{1i}\textbf{\emph{g}}_{1}+k_{2i}\textbf{\emph{g}}_2+\cdots+
k_{ti}\textbf{\emph{g}}_{t}$ for $1\leq i\leq s$. Thus, $I_s=K^TG^T$

\quad {\rm \textbf{RETURN}} $L:= K^TH^T$

\end{description}}}
\end{center}

\begin{example}
Let $A=\sigma(\mathbb{Q})\langle x, y \rangle$ defined through the relation $yx=-xy+1$.
Given the matrix
\begin{center}
$F=\begin{bmatrix} 1 &1 \\ xy&0 \\x^2&0\\ 1&y \end{bmatrix},$
\end{center}
we apply the above algorithm in order to verify if $F$ has a left inverse. For this, we compute a
Gröbner basis of the left module generated by the rows of $F$. Considering the deglex order on
$Mon(A)$, with $x\succ y$, and the TOPREV order on $Mon(A^2)$, with
$\textbf{\emph{e}}_1>\textbf{\emph{e}}_2$, a Gröbner basis for $_A\langle F^T\rangle$ is
$\{\textbf{\emph{e}}_1, \textbf{\emph{e}}_2\}$. In consequence, $F$ has a left inverse and, from calculations obtained during
the process of Buchberger's algorithm, we have that
\begin{center}
$L=\begin{bmatrix}xy^2-y&y+1&0&-xy+1\\ -xy^2+y+1 &-y-1&0&xy-1 \end{bmatrix}$
\end{center}
is a left inverse for $F$.
\end{example}

\begin{corollary}\label{SquareInverse}
Let $F$  be a square matrix of size $r\times r$ with entries in a ring $S$. Then, $F$ is invertible
if and only if the rows of $F$ conform a basis of $S^s$.
\begin{proof}
Let $L\in M_{r}(A)$ such that $LF=I_r=FL$. From $LF=I_r$ it follows that the rows of $F$ generate
$S^r$. Let $f^t$ and $l^t$ be like in the proof of Proposition \ref{LeftInverse}; since $FL=I_r$,
then $l^t \circ f^t = i_{S^r}$ and, therefore, $f^t$ is a monomorphism, i.e., $Syz(F^T)=0$. Thus,
the rows of $F$ are linearly independent, and this complete the first implication. Conversely,
since the rows of $F$ generate $S^r$, by Proposition \ref{LeftInverse}, $F$ has a left inverse. Let
$L$ be a such inverse, then $LF=I_r$. We have $FLF=F$, this implies that $(FL-I_r)F=0_r$, but
$Syz(F^T)=0$, then $FL=I_r$, i.e., $F^{-1}=L$.
\end{proof}
\end{corollary}
\begin{corollary}
Let $A$ be a bijective skew $PBW$ extension and $F\in M_{r}(A)$ a square matrix over $A$. The
algorithm below determines  whether $F$ is invertible, and in the positive case, it computes the inverse
of $F$:

\begin{center}
\fbox{\parbox[c]{12cm}{
\begin{center}
{\rm \textbf{Algorithm for the inverse of a square matrix}}
\end{center}
\begin{description}
\item[]{\rm \textbf{INPUT}:} A square matrix $F\in M_{r}(A)$
\item[]{\rm \textbf{OUTPUT}:} A matrix $L\in M_{r}(A)$ satisfying $LF=I_r=FL$ if it exists, and $0$ in other case
\item[]{\rm \textbf{INITIALIZATION}:}

Use the algorithm in Corollary \ref{LeftInverseAlg} to determine if $F$ is left invertible

{\rm \textbf{IF}} $F$ is not left invertible\\
\quad {\rm \textbf{RETURN}} $0$\\

{\rm \textbf{ELSE}} Compute $Syz(F^T)$

\quad\quad{\rm \textbf{IF}} $Syz(F^T)\neq 0$

\quad\quad {\rm \textbf{RETURN}} $0$

\quad\quad {\rm \textbf{ELSE}} Compute the matrices $H$ and $K$ in the algorithm of Corollary \ref{LeftInverseAlg}\\

{\rm \textbf{RETURN}} $L:=K^TH^T$

\end{description}}}
\end{center}
\end{corollary}

\begin{example}
For this example, we consider the \textit{additive analogue of the Weyl algebra}. Recall that, if  $\Bbbk$
is a field, the $\Bbbk$-algebra $A_n(q_1,\dots,q_n)$ is generated by
$x_1,\dots,x_n,y_1,\dots,y_n$ and subject to the relations:
\begin{center}
$x_jx_i = x_ix_j, y_jy_i = y_iy_j, \ 1 \leq i,j \leq n$,

$y_ix_j=x_jy_i, \ i\neq j$,

$y_ix_i = q_ix_iy_i + 1, \ 1\leq i\leq n$,
\end{center}
where $q_i\in \Bbbk-\{0\}$. It is not difficult to show that
$A_n(q_1, \dots, q_n)$ is a bijective skew $PBW$ extension. We take, $n=2$, $q_1=\frac{1}{2}$
and $q_2=\frac{1}{3}$; on $Mon(A)$ we consider the deglex order and, oven $Mon(A^2)$ the TOPREV order with
$\textbf{\emph{e}}_1> \textbf{\emph{e}}_2$. Let $F$ be the following matrix
\begin{center}
$F=\begin{bmatrix}x_1y_1^2 &x_2y_2^2 \\x_2y_2 &x_1y_1\end{bmatrix}$.
\end{center}
We want to check if columns of $F$ conform a basis for $A^2$, and we know
that this is true if and only if $F^T$ is invertible. Using the above algorithm, we start verifying
if $F^T$ has a left inverse; for this purpose, we compute a Gröber basis of the left $A$-module
generated by the rows of $F^T$, i.e., of the left $A$-module $Im(F)$. Using the Buchberger's algorithm for modules, it is possible to show that $G=\{\boldsymbol{f}_1, \boldsymbol{f}_2, \boldsymbol{f}_3\}$ is a Gröbner
basis for this module, where
$\textbf{\emph{f}}_1=x_1y_1^2\textbf{\emph{e}}_1+x_2y_2\textbf{\emph{e}}_2$,
$\textbf{\emph{f}}_2=x_2y_2^2\textbf{\emph{e}}_1+x_1y_1\textbf{\emph{e}}_2$ and
$\textbf{\emph{f}}_3=-\frac{1}{4}x_1^2y_1^3\textbf{\emph{e}}_2+\frac{1}{9}x_2^2y_2^3\textbf{\emph{e}}_2-
\frac{3}{2}x_1y_1^2\textbf{\emph{e}}_2+ \frac{4}{3}x_2y_2^2\textbf{\emph{e}}_2$.
Applying the division algorithm, we can check that $\textbf{\emph{e}}_1\notin \langle
G \rangle$, therefore  $_A\langle G \rangle\neq A^2$. Thus $F^T$  has no a left inverse and, hence, the
columns of $F$ are not a basis for $A^2$.
\end{example}

\begin{remark}
If $S$ is a left (or right) Noetherian ring, then every epimorphism $\alpha:S^{r}\to S^{r}$ is an
automorphism (see Proposition 1.14 in \cite{Lam1}). This implies that
every left (or right) Noetherian ring is $\mathcal{WF}$ (see \cite{Cohn1}). Therefore, to test if $F\in M_{r}(S)$ is
invertible, it is enough to show that $F$ has a right or a left inverse. So, in the above
algorithm, when $A$ is a bijective $PBW$ extension of a $LGS$ ring, it is not necessary the
computation of $Syz_{S}(F^T)$ to test whether the matrix is invertible, it would be sufficient to
apply the algorithm for the left inverse given in Corollary \ref{LeftInverseAlg}.
\end{remark}

Now we will consider the right inverse of a rectangular matrix. We start with the following
theoretical result.

\begin{proposition}\label{RightInverse}
Let $F$ be a rectangular matrix of size $r\times s$ with entries in the ring $S$. If $F$ has right
inverse, then $s\geq r$ and the module of syzygies of the submodule generated by the rows of $F$ is
zero, i.e., $Syz(F^T)=0$. In other words, if $F$ has a right inverse then the rows of $F$ are
linearly independent.
\begin{proof}
$s\geq r$ since we are assuming that $S$ is $\mathcal{RC}$ (Proposition \ref{2.1.4}).
Let $L\in M_{s\times r}(S)$ such that $FL=I_r$. Consider the homomorphisms $f^{t}$
and $l^t$ as in Proposition \ref{LeftInverse}, then $f^t$ is a monomorphism. Hence, $\ker(f^t)=0$,
i.e., $Syz(F^T)=0$.
\end{proof}
\end{proposition}

\begin{proposition}\label{RightInverse}
Let $F$ be a rectangular matrix of size $r\times s$ with entries in the ring $S$. If $F$ has right
inverse, then $s\geq r$. Moreover, $F$ has a right inverse if and only if $Syz(F^T)=\textbf{0}$ and $Im(F^T)$
is a summand direct of $S^{s}$, where $Im(F^T)$ denotes the
module generated by the columns of $F^T$ i.e., the module generated by the rows
of $F$.
\begin{proof}
To begin, $s\geq r$ since we are assuming that $S$ is $\mathcal{RC}$ (Proposition \ref{2.1.4}).
Now, let $L\in M_{s\times r}(S)$ such that $FL=I_r$. Consider the
homomorphisms $f^{t}$ and $l^t$ as in Proposition \ref{LeftInverse}, then $l^{t}\circ
f^{t}=i_{S^r}$, i.e,  $f^t$ is a split monomorphism. Thus, $S^s=Im(f^{t})\oplus \ker(l^{t})$, and
$Im(f^{t})$ is a direct summand of $S^{s}$. Conversely, let $M$ be a submodule of $S^s$ such that
$S^s=Im(f^t)\oplus M$. So, given $\textbf{\emph{f}}\in S^s$ there exist unique elements
$\textbf{\emph{f}}_1\in Im(f^t)$ and $\textbf{\emph{f}}_2\in M$ such that
$\textbf{\emph{f}}=\textbf{\emph{f}}_1+\textbf{\emph{f}}_2$. Define the homomorphism $l^t:S^s \to
S^r$ as $l^t(\textbf{\emph{f}}):=\textbf{\emph{h}}_{\textbf{\emph{f}}}$, where
$\textbf{\emph{h}}_{\textbf{\emph{f}}}\in S^r$ is such that
$f^t(\textbf{\emph{h}}_{\textbf{\emph{f}}})=\textbf{\emph{f}}_1$. By hypothesis, $Syz(F^{T})=\textbf{0}$, so
$f^{t}$ is injective and we get that $l^t$ is well defined. It is not difficult to show that $l^{t}$ is
a $S$-homomorphism. Consequently, $l^t\circ f^t=i_{S^r}$ and if
$L^T:=m(l^t)$, then $FL=I_r$, i.e., $F$ has a right inverse.
\end{proof}
\end{proposition}
\begin{remark}
If we had a computational tool for to check if a submodule of a free module is a summand direct,
then Proposition \ref{RightInverse} would establish an algorithm to check the existence of a right inverse.
\end{remark}

Following \cite{Chyzak} and \cite{Quadrat}, consider a matrix $F:=[f_{ij}]\in M_{r\times s}(A)$,
with $s\geq r$, where $A$ is a bijective
skew $PBW$ extension endowed with an involution $\theta$, i.e., a function $\theta:S\to S$ such
that $\theta(a+b)=\theta(a)+\theta(b), \theta(ab)=\theta(b)\theta(a)$ and $\theta^2=i_S$, for all
$a,b\in S$. Note that $\theta(1)=1$, and hence, $\theta$ is an anti-isomorphism of $S$. We define
$\theta(F):=[\theta(f_{ij})]$. Observe that if $K\in M_{s\times r}(A)$, then
\begin{equation}\label{equ14.1.2}
\theta(FK)^T=\theta(K)^T\theta(F)^T.
\end{equation}
\begin{proposition}
Let $A$ be a bijective skew $PBW$ extension endowed with an involution $\theta$ and let
$F:=[f_{ij}]\in M_{r\times s}(A)$, with $s\geq r$. Then, $F$ has a right inverse if and only if for
each $1\leq j\leq r$, $\textbf{e}_j\xrightarrow{G'}_+ \textbf{\emph{0}}$, where $G'$ is a Gröbner
basis of the left $A$-module generated by the columns of $\theta(F)$ and $\{\textbf{e}_j\}_{j=1}^r$
is the canonical basis of $A^r$.
\end{proposition}
\begin{proof}
$G:=[g_{ij}]\in M_{s\times r}(A)$ is a right inverse of $F$ if and only if $FG=I_r$, and this is
equivalent to say that
\begin{center}
$\textbf{\emph{e}}_j=\begin{bmatrix}f_{11}\\ f_{21}\\ \vdots
\\ f_{r1}\end{bmatrix}\cdot g_{1j}+\cdots +\begin{bmatrix}f_{1s}\\ f_{2s}\\ \vdots
\\ f_{rs}\end{bmatrix}\cdot g_{sj}$, $1\leq j\leq r$;
\end{center}
applying $\theta$ we obtain
\begin{center}
$\textbf{\emph{e}}_j=\theta(g_{1j})\cdot \begin{bmatrix}\theta(f_{11})\\ \theta(f_{21})\\
\vdots
\\ \theta(f_{r1})\end{bmatrix}+\cdots +\theta(g_{sj})\cdot \begin{bmatrix}\theta(f_{1s})\\ \theta(f_{2s})\\ \vdots
\\ \theta(f_{rs})\end{bmatrix}$.
\end{center}
Thus, $G$ is a right inverse of $F$ if and only if the canonical vectors of $A^r$ belong to the
left $A$-module generated by the columns of $\theta(F)$, i.e.,
$\textbf{\emph{e}}_1,\dots,\textbf{\emph{e}}_r\in \langle \theta(F)\rangle$. Let $G'$ be a Gröbner
basis of $\langle \theta(F)\rangle$, then  $G$ is a
right inverse of $F$ if and only if for each $j$, $\textbf{\emph{e}}_j\xrightarrow{G'}_+
\textbf{0}$.
\end{proof}

\begin{corollary}
Let $A$ be a bijective skew $PBW$ extension and $F\in M_{r\times s}(A)$ be a rectangular matrix
over $A$. The algorithm below determines if $F$ is right invertible, and in the positive case, it
computes the right inverse of $F$:
\end{corollary}

\begin{center}
\fbox{\parbox[c]{12cm}{
\begin{center}
{\rm \textbf{Algorithm 1 for the right inverse of a matrix}}
\end{center}
\begin{description}
\item[]{\rm \textbf{INPUT}:} An involution $\theta$ of $A$; a rectangular matrix $F\in M_{r\times s}(A)$
\item[]{\rm \textbf{OUTPUT}:} A matrix $H\in M_{s\times r}(A)$ satisfying $FH=I_r$ if it exists, and $0$ in other case
\item[]{\rm \textbf{INITIALIZATION}:}

{\rm \textbf{IF}} $s<r$

\quad {\rm \textbf{RETURN}} $0$

\quad \quad  {\rm \textbf{IF}} $s\geq r$, let
$G':=\{\textbf{\emph{g}}_1,\ldots,\textbf{\emph{g}}_t\}$ be a Gröbner basis for the left submodule
generated by columns of $\theta(F)$ and let $\{\textbf{\emph{e}}_j\}_{j=1}^r$ be the canonical basis of
$A^r$. Use the division algorithm to verify if $\textbf{\emph{e}}_{j}\in \langle G'\rangle$ for
each $1\leq j \leq r$.

{\rm \textbf{IF}} there exists some $\textbf{\emph{e}}_{j}$ such that $\textbf{\emph{e}}_j\notin
\langle G' \rangle$,

\quad {\rm \textbf{RETURN}} $0$

{\rm \textbf{IF}} $\langle G'\rangle=A^r$, let $J\in M_{s\times t}(A)$ with the property
$G'^T=J^T\theta(F)^T$, and consider $K:=[k_{ij}]\in M_{t\times r}$, where the $k_{ij}$'s are such
that $\textbf{\emph{e}}_{j}=k_{1j}\textbf{\emph{g}}_{1}+k_{2j}\textbf{\emph{g}}_2+\cdots+
k_{tj}\textbf{\emph{g}}_{t}$ for $1\leq j\leq r$. Thus, $I_r=K^TG'^T$

\quad {\rm \textbf{RETURN}} $H:= \theta(J)\theta(K)$

\end{description}}}
\end{center}
\begin{proof}
Applying (\ref{equ14.1.2}) we get
\begin{center}
$I_r=K^TG'^T=K^TJ^T\theta(F)^T=\theta(\theta(K))^T\theta(\theta(J))^T\theta(F)^T=\theta(\theta(J)\theta(K))^T\theta(F)^T=
\theta(F\theta(J)\theta(K))^T$,
\end{center}
so $I_r=\theta(F\theta(J)\theta(K))=\theta(I_r)$, and from this we get $I_r=F\theta(J)\theta(K)$.
\end{proof}

\begin{example}
Let us consider the ring $A=\sigma(\mathbb{Q})\langle x,y\rangle$, with $yx=-xy+1$. Using the above
algorithm, we will compute a right inverse for
\begin{center}
$F=\begin{bmatrix}x&0&1\\y-1&x-1&x-y\end{bmatrix}$
\end{center}
provided that it exists. For this, we consider the involution $\theta$ on $A$ given by
$\theta(x)=-x$ and $\theta(y)=-y$. With this involution, we have that $\theta(xy)=-xy+1$.
Thus,
\begin{center}
$\theta(F)=\begin{bmatrix}-x&0&1\\-y-1&-x-1&-x+y\end{bmatrix}$
\end{center}
We start computing a Gröbner basis for the left module generated by the columns of
$\theta(F)$. We get that $G'=\{\textbf{\emph{e}}_1,\textbf{\emph{e}}_2\}$ is a
Gröbner basis for $_A\langle \theta(F)\rangle$. In this case, $F$  has a right inverse and
\begin{center}
$J=\begin{bmatrix}-x+y&-1\\ x^2+2xy-y^2-x+y-1 & x+y-1 \\ -x^2-xy+2 & -x\end{bmatrix}$ is such that $G'^T=J^T\theta(F)^T$.
\end{center}
Since $G'^T=I_2$, then $K=I_2$ and $L:=\theta(J)$ is a right inverse for $F$, where
\begin{center}
$\theta(J)=\begin{bmatrix}x-y&-1\\ x^2-2xy-y^2+x-y+1 & -x-y-1 \\ -x^2+xy+1 & x\end{bmatrix}$.
\end{center}
\end{example}

To find involutions of skew $PBW$ extensions it is a difficult task, so the above algorithm is not
practical. A second algorithm for testing the existence and computing a right inverse of a matrix
uses the theory of Gröbner bases for right modules developed in \cite{Gallego6}. For this
we will made a simple adaptation of Proposition \ref{LeftInverse} and Corollary
\ref{LeftInverseAlg} for right submodules, using the right notation.

\begin{proposition}
Let $F$ be a rectangular matrix of size $r\times s$ with entries in a ring $S$. If $F$ has right
inverse, then $s\geq r$. Moreover, $F$ has a right inverse if and only if the right module generated by
the columns of $F$ coincides with $S^r$.
\end{proposition}
\begin{proof}
The first statement follows from Proposition \ref{2.1.4}. Now, suppose that
$F$ has a right inverse and let $L$ be a matrix such that $FL=I_r$. Define the following
homomorphism of right free $S$-modules:
\begin{equation*}
\begin{aligned}
f: S^s &\to S^r   \\
\textbf{\emph{a}} &\mapsto F\textbf{\emph{a}}
\end{aligned}
\qquad \qquad
\begin{aligned}
l:S^r &\to S^s\\
\textbf{\emph{b}}&\mapsto L\textbf{\emph{b}},
\end{aligned}
\end{equation*}
then $m(f)=F$ and $m(l)=L$. Whence, $m(f\circ l)=FL=I_r$, i.e, $f$ is an epimorphism.
Therefore, $Im(f)=S^r$, i.e., the right
submodule generates by columns of $F$ coincides with the free module $S^r$. Conversely,  if
$Im(F)=S^r$, then for $f$ defined as above, there exist $\textbf{\emph{a}}_{1}\ldots,
\textbf{\emph{a}}_s\in S^s$ such that $f(\textbf{\emph{a}}_i)=\textbf{\emph{e}}_{i}$ for each
$1\leq i\leq s$, and where $\textbf{\emph{e}}_{1},\ldots, \textbf{\emph{e}}_s$ denote the canonical
vectors of $S^s$. Thus, if $\textbf{\emph{a}}_{j}=\begin{bmatrix} a_{1j}& a_{2j} &\cdots& a_{rj}
\end{bmatrix}^T$, we have
\[F\textbf{\emph{a}}_{j}=F\begin{bmatrix} a_{1j}& a_{2j} &\cdots& a_{rj} \end{bmatrix}=F^{(1)}a_{1j}+\cdots+ F^{(r)}a_{rj}=\textbf{\emph{e}}_j,\]
where $F^{(j)}$ denotes the $j$-th column of $F$, $1\leq j\leq r$. So, if $L$ is the matrix
whose columns are the vectors $\textbf{\emph{a}}_{j}^T$, then $FL=I_r$, i.e., $F$ has a right inverse.
\end{proof}
Thus, considering the results about Gröbner bases for right modules (see \cite{Gallego6}), we have the following alternative
algorithm for testing  the existence of a right inverse.
\begin{corollary}\label{RightInverseAlg}
Let $A$ be a bijective skew $PBW$ extension and $F\in M_{r\times s}(A)$ be a rectangular matrix
over $A$. The algorithm below determines if $F$ is right invertible, and in the positive case, it
computes a right inverse of $F$:
\end{corollary}

\begin{center}
\fbox{\parbox[c]{12cm}{
\begin{center}
{\rm \textbf{Algorithm 2 for the right inverse of a matrix}}
\end{center}
\begin{description}
\item[]{\rm \textbf{INPUT}:} A rectangular matrix $F\in M_{r\times s}(A)$
\item[]{\rm \textbf{OUTPUT}:} A matrix $L\in M_{s\times r}(A)$ satisfying $FL=I_r$ if it exists, and $0$ in other case
\item[]{\rm \textbf{INITIALIZATION}:}

{\rm \textbf{IF}} $s< r$

\quad {\rm \textbf{RETURN}} $0$

\quad \quad  {\rm \textbf{IF}} $s\geq r$, let
$G:=\{\textbf{\emph{g}}_1,\ldots,\textbf{\emph{g}}_t\}$ be a right Gröbner basis for the right submodule
generated by columns of $F$ and let $\{\textbf{\emph{e}}_j\}_{j=1}^r$ be the canonical basis of $A_{A}^r$. Use
right version of division algorithm to verify if $\textbf{\emph{e}}_{i}\in \langle G\rangle_A$ for each $1\leq i
\leq r$.

{\rm \textbf{IF}} there exists some $\textbf{\emph{e}}_{j}$ such that $\textbf{\emph{e}}_j\notin
\langle G \rangle_A$,

\quad {\rm \textbf{RETURN}} $0$

{\rm \textbf{IF}} $\langle G\rangle_A=A^r$, let $H\in M_{s\times t}(A)$ with the property $G=FH$,
and consider $K:=[k_{ij}]\in M_{t\times s}$, where the $k_{ij}$'s are such that
$\textbf{\emph{e}}_{j}=\textbf{\emph{g}}_{1}k_{1j}+\textbf{\emph{g}}_2k_{2j}+\cdots+
\textbf{\emph{g}}_{t}k_{tj}$ for $1\leq i\leq r$. Thus, $I_r=GK$

\quad {\rm \textbf{RETURN}} $L:= HK$

\end{description}}}
\end{center}

\begin{example}\label{rightinverse2}
Consider the ring $A=\sigma(\mathbb{Q})\langle x,y\rangle$, with $yx=-xy+1$, and let $F$ be the
matrix given by
\begin{center}
$F=\begin{bmatrix}y^2&-xy&y\\xy-1&x^2&x\end{bmatrix}$.
\end{center}
Applying the right versions of Buchberger's algorithm, we have
that a Gröbner basis for the right module generated by the columns of $F$ is
$G=\{\textbf{\emph{e}}_1,\textbf{\emph{e}}_2\}$. From Corollary \ref{RightInverseAlg} we can show
that $F$ has a right inverse; moreover, one right inverse for $F$ is given by
\begin{center}
$L=\begin{bmatrix}0&-1\\-1 &0 \\ x& y \end{bmatrix}$.
\end{center}
\end{example}

\section{Computing the projective dimension}\label{10.10}
\noindent
Given $M$ an $S$-module and
\begin{equation}\label{projectiveresolution}
\cdots \xrightarrow{f_{r+1}} P_r\xrightarrow{f_r}
P_{r-1}\xrightarrow{f_{r-1}} \cdots \xrightarrow{f_2}
P_1\xrightarrow{f_1} P_0\xrightarrow{f_0} M\longrightarrow 0
\end{equation}
a projective resolution of $M$, it is not difficult to show that if $r$ is the smallest integer
such $Im(f_r)$ is projective, then $r$ does not depend on the resolution and $pd(M)=r$ (c.f.
\cite{Gallego6}, Theorem 2.4.2). Therefore, we can consider a free resolution $\{f_i\}_{i\geq 0}$,
which we can calculate using the some of the applications of Gröbner bases theory. Hence, by
Theorem \ref{6.3.3} we obtain the following algorithm which computes the projective dimension of a
module $M\subseteq A^m$ given by a finite set of generators, where $A$ is a bijective skew $PBW$
extension of a $LGS$  ring $R$ (\textit{left Gröbner soluble}, see \cite{Gallego6} and
\cite{Gallego5}) with finite left global dimension. Note that $A$ is left Noetherian (Theorem
\ref{1.3.4}) and ${\rm lgld}(A)<\infty$ (see \cite{lezamareyes1}).
\begin{center}
\fbox{\parbox[c]{12cm}{
\begin{center}
{\rm \textbf{Projective dimension of a module\\ over
a bijective skew $PBW$ extension\\
Algorithm 1}}
\end{center}
\begin{description}
\item[]{\rm \textbf{INPUT}:} $\text{\rm lgld}(A)<\infty, M=\langle \textbf{\emph{f}}_1,\dots ,\textbf{\emph{f}}_s\rangle\subseteq
A^m$, with $\textbf{\emph{f}}_k\neq 0, \\ 1\leq k\leq s$
\item[]{\rm \textbf{OUTPUT}:} $\text{\rm pd}(M)$

\smallskip
\item[]{\rm \textbf{INITIALIZATION}:} Compute a free resolution $\{f_i\}_{i\geq 0}$ of $M$

$i:=0$
\item[]{\rm \textbf{WHILE}} $i\leq \text{\rm lgld}(A)$ {\rm \textbf{DO}}

{\rm \textbf{IF}} $Im(f_i)$ is projective {\rm \textbf{THEN}} $\text{\rm pd}(M)=i$

\item[]\begin{quote}{\rm \textbf{ELSE}} $i=i+1$
\end{quote}

\end{description}}}
\end{center}
Observe that, in the previous algorithm, we no need to compute finite free resolutions of $M$; any
free resolution computed using syzygies is enough.\\
\\
Next, we present  another algorithm for computing the left projective dimension of a module
$M\subseteq A^m$ given by a finite free resolution:
\begin{equation}\label{equ12.12.1}
0\to A^{s_m}\xrightarrow{f_{m}} A^{s_{m-1}}\xrightarrow{f_{m-1}} A^{s_{m-2}}\xrightarrow{f_{m-2}}
\cdots \xrightarrow{f_2} A^{s_1}\xrightarrow{f_1} A^{s_0}\xrightarrow{f_0} M\longrightarrow 0.
\end{equation}
This algorithm is supported by Corollary \ref{6.3.4}.
\begin{center}
\fbox{\parbox[c]{13cm}{
\begin{center}
{\rm \textbf{Projective dimension of a module\\ over a bijective skew $PBW$ extension\\ Algorithm
2}}
\end{center}
\begin{description}
\item[]{\rm \textbf{INPUT}:} An $A$-module $M$ defined by a finite
free resolution (\ref{equ12.12.1})
\item[]{\rm \textbf{OUTPUT}:} $\text{\rm pd}(M)$

\smallskip
\item[]{\rm \textbf{INITIALIZATION}:} Set $j:=m$ and $H_j:=F_m$,
with $F_m$ the matrix of $f_m$ in the canonical bases

\item[]{\rm \textbf{WHILE}} $j\leq m$ {\rm \textbf{DO}}

\textit{Step} 1.  Check whether or not $H_j^T$ admits a right inverse $G_j^T$:
\begin{quote}

(a) If no right inverse of $H_j^T$ exists, then $\text{\rm pd}(M)=j$

(b) If there exists a right inverse $G_j^T$ of $H_j^T$ and
\begin{quote}
(i) If $j=1$, then $\text{\rm pd}(M)=0$

(ii) If $j=2$, then compute (\ref{equ6.3.6})

(iii) If $j\geq 3$, then compute (\ref{equ6.3.5})
\end{quote}
\end{quote}
\textit{Step} 2. $j:=j-1$
\end{description}}}
\end{center}

\begin{example}\label{projdim}
Let $A$ be the ring $\sigma(\mathbb{Q})\langle x,y\rangle$, where $yx=xy+x$. We will calculate the
projective dimension of the left module $M=\,_A\langle (1,1), (xy,0),(y^2,0), (0,x) \rangle$.
For this, we use the deglex order on $Mon(A)$, with $x\succ y$, and the TOP order over
$Mon(A^2)$, with $\textbf{\emph{e}}_2>\textbf{\emph{e}}_1$. Using Gröbner bases, it is possible to show that a free resolution
for $M$ is given by:
\begin{equation*}
\begin{diagram}
\node{0} \arrow{e} \node{A} \arrow{e,t}{F_2} \node{A^3} \arrow{e,t}{F_1} \node{A^4}
\arrow{e,t}{F_0} \node{M} \arrow{e} \node{0}
\end{diagram}
\end{equation*}
where,
\begin{center}
$F_0=\begin{bmatrix} 1&xy&y^2&0\\ 1&0&0&x \end{bmatrix}$,
$F_1=\begin{bmatrix} 0&-xy&xy^2+2xy\\ -y+1&1&-y-1\\ x&0&0\\0&y-1&1-y^2 \end{bmatrix}$,
$F_2=\begin{bmatrix} 0\\ y+1\\ 1 \end{bmatrix}$.
\end{center}
In order to apply the above algorithm, we start checking whether $F_2=\begin{bmatrix} 0 & y+1& 1
\end{bmatrix}^T$ has a right inverse. A straightforward calculation shows that a right inverse for
$F_2$ is $G_2=\begin{bmatrix} 0 & 1& -y \end{bmatrix}^T$, so we compute (\ref{equ6.3.6}):
\begin{equation}
\begin{diagram}
\node{0} \arrow{e} \node{A^3} \arrow{e,t}{H_1} \node{A^5} \arrow{e,t}{H_0} \node{M} \arrow{e} \node{0}
\end{diagram}
\end{equation}
where
\begin{center}
$H_1:=\begin{bmatrix}0&-xy&xy^2+2xy \\ -y+1&1&-y-1\\ x&0&0\\0&y-1&1-y^2\\0&1&-y\end{bmatrix}$ and
$H_0:=\begin{bmatrix} 1&xy&y^2&0\\ 1&0&0&x \end{bmatrix}$.
\end{center}
To verify if $H_1^T$ has a right inverse, we must calculate a Gröbner basis for the right module
generated by the columns of $H_1^T$. Since the ring $A$ considered is a bijective skew $PBW$
extension, we can use the right version of Buchberger's algorithm. For this, we consider the deglex
order on $Mon(A)$, with $x\succ y$, and the TOP order over $Mon(A^3)$, with
$\textbf{\emph{e}}_1<\textbf{\emph{e}}_2<\textbf{\emph{e}}_3$. Applying this algorithm, we obtain the following Gröbner basis for $\langle
H_1^T\rangle_A$, $G=\{(x,0,0), (1-y,0,-1), (0,-1,1), (0,-x,0), (0,y-1,0)\}$. Note that $\textbf{\emph{e}}_1$ is not
reducible by $G$, thus $\langle G\rangle_A\neq A^3$ and hence $H_1^T$ does not have a right inverse.
Therefore, pd$(M)=1$.
\end{example}

\begin{remark}
The above algorithms can be used for testing if a given module $M$ is projective: we can compute
its projective dimension, and hence, $M$ es projective if and only if ${\rm pd}(M)=0$.
\end{remark}

\section{Test for stably-freeness}
\noindent
Theorem \ref{3.2.19a} gives a procedure for testing stably-freeness for a module
$M\subseteq A^m$ given by an exact sequence
\begin{center}
$0\rightarrow A^s\xrightarrow{f_1}A^r\xrightarrow{f_0}M\rightarrow 0$,
\end{center}
where $A$ is a bijective skew $PBW$ extension.

\begin{center}
\fbox{\parbox[c]{12cm}{
\begin{center}
{\rm \textbf{Test for stably-freeness\\
Algorithm 1}}
\end{center}
\begin{description}
\item[]{\rm \textbf{INPUT}:} $M$ an $A$-module with exact sequence
\begin{equation*}
0\rightarrow A^s\xrightarrow{f_1}A^r\xrightarrow{f_0}M \rightarrow 0
\end{equation*}

\item[]{\rm \textbf{OUTPUT}:} TRUE if $M$ is stably free, FALSE otherwise

\item[]{\rm \textbf{INITIALIZATION}:} Compute the matrix $F_1$ of $f_1$

\item[]{\rm \textbf{IF}} $F_1^T$ has right inverse {\rm \textbf{THEN}}
\begin{quote}
\textbf{RETURN} TRUE
\end{quote}
\item[]{\rm \textbf{ELSE}}
\begin{quote}
\textbf{RETURN} FALSE
\end{quote}
\end{description}}}
\end{center}

\begin{example}
Let $A=\sigma(\mathbb{Q})\langle x,y\rangle$, with $yx=-xy$. We want to know if the left $A$-module
$M$ given by
\begin{center}
$M=\,_A\langle \textbf{\emph{e}}_3+\textbf{\emph{e}}_1,\textbf{\emph{e}}_4+\textbf{\emph{e}}_2,x\textbf{\emph{e}}_2+
x\textbf{\emph{e}}_1,y\textbf{\emph{e}}_1,y^2\textbf{\emph{e}}_4,x\textbf{\emph{e}}_4+y\textbf{\emph{e}}_3 \rangle$
\end{center}
is stably free or not. To answer this question, we start computing a finite presentation of $M$.
Considering the deglex order on $Mon(A)$ with $x\succ y$, the TOP order on $Mon(A^4)$ with
$\textbf{\emph{e}}_4>\textbf{\emph{e}}_3>\textbf{\emph{e}}_2>\textbf{\emph{e}}_1$, and using the methods established
in the previous sections, we have that a system of generators for $Syz(M)$ is given by
\[S=\{(0,-xy^2,y^2,-xy,x,0), (-y^2,xy,y,x+y,0,y), (y^3,0,0,-y^2,x,-y^2)\}.\]
Therefore, we get a following finite presentation for $M$:
\begin{equation}\label{PresentationM}
\begin{diagram}
\node{A^3} \arrow{e,t}{F_1} \node{A^6} \arrow{e,t}{F_0} \node{M} \arrow{e} \node{0}
\end{diagram}
\end{equation}
where,
\begin{center}
$F_1:=\begin{bmatrix}0 &-y^2&y^3 \\ -xy^2&xy&0\\ y^2&y&0\\-xy&x+y&-y^2\\x&0&x\\0&y&-y^2\end{bmatrix}$ and
$F_0:=\begin{bmatrix} 1&0&x&y&0&0\\0&1&x&0&0&0\\1&0&0&0&0&y\\0&1&0&0&y^2&x \end{bmatrix}$.
\end{center}
Applying  the method for computing the syzygy module, we have that $Syz_{A}(F_1)=0$, so the
presentation obtained in \ref{PresentationM} becomes
\begin{equation*}
\begin{diagram}
\node{0} \arrow{e} \node{A^3} \arrow{e,t}{F_1} \node{A^6} \arrow{e,t}{F_0} \node{M} \arrow{e} \node{0}
\end{diagram}
\end{equation*}

Finally, we must to test if $F_1^T$ has a right inverse. For this, we calculate a Gröbner basis for
the right module generated by the columns of $F_1^T$. Using the TOP order on $Mon(A^3)$, with
$\textbf{\emph{e}}_3>\textbf{\emph{e}}_2>\textbf{\emph{e}}_1$, a Gröbner basis for $\langle
F_1^T\rangle_A$ is given by $G=\{\textbf{\emph{f}}_i\}_{i=1}^7$, where $\textbf{\emph{f}}_i$ is the
$i$-th column of $F_1^T$ for $1\leq i\leq 6$, and
$\textbf{\emph{f}}_7=-\textbf{\emph{e}}_2xy^2+\textbf{\emph{e}}_1xy^2$. Note that,
for example, $\textbf{\emph{e}}_{1}\notin \langle G \rangle_{A}$ so that $A^6\neq \langle
G\rangle_A$. Thus, $F_1^T$ has not right inverse and hence $M$ is not stably free.
\end{example}

\begin{remark}
From Theorem \ref{3.2.19a}, if $M$ is a left $A$-module with exact sequence
$0\rightarrow A^s\xrightarrow{f_1}A^r\xrightarrow{f_0}M \rightarrow 0$, then
$M^T\cong Ext_{A}^1(M,A)$, where $M^T=S^s/Im(f_1^{T})$ and $f_1^T:S^r\rightarrow
S^s$ is the homomorphism of right free $S$-modules induced by the
matrix $F_1^T$. Thus, for testing if $M$ is stably free, we can use the results
of Section 5.6 in \cite{Gallego6} and computing a Gröbner basis for the right module
generated by columns of $F_1^{T}$. Using the right version of the division algorithm,
is possible to check whether $S^s=Im(F_{1}^{T})$. If this last equality holds, then $M^{T}=0$ and $M$
is stably free.
\end{remark}

Corollary \ref{6.3.4} gives another procedure for testing stably-freeness for a module $M\subseteq
A^m$ given by a finite free resolution (\ref{equ6.2.11}) with $S=A$: Indeed, if $m\geq 3$ and $f_m$
has not left inverse, then $M$ is non stably free; if $f_m$ has a left inverse we compute then the
new finite free resolution (\ref{equ6.3.5}) and we check if $h_{m-1}$ has a left inverse. We can
repeat this procedure until (\ref{equ6.3.6}); if $h_1$ has not left inverse, then $M$ is non stably
free. If $h_1$ has a left inverse, then $M$ is stably free.

\begin{example}\label{projdim2}
Let $A$ be the ring $\sigma(\mathbb{Q})\langle x,y\rangle$, where $yx=xy+x$ and consider the left
module $M=\,_A\langle (1,1), (xy,0),(y^2,0), (0,x) \rangle$ given in the Example \ref{projdim}.
As we saw there, a finite presentation for $M$ is given by:
\begin{equation}
\begin{diagram}
\node{0} \arrow{e} \node{A^3} \arrow{e,t}{H_1} \node{A^5} \arrow{e,t}{H_0} \node{M} \arrow{e} \node{0}
\end{diagram}
\end{equation}
where
\begin{center}
$H_1:=\begin{bmatrix}0&-xy&xy^2+2xy \\ -y+1&1&-y-1\\ x&0&0\\0&y-1&1-y^2\\0&1&-y\end{bmatrix}$ and
$H_0:=\begin{bmatrix} 1&xy&y^2&0\\ 1&0&0&x \end{bmatrix}$.
\end{center}
In such example, we showed that $H_1^T$ has not a right inverse, hence $M$ is not a stably free
module.
\end{example}

\section{Computing minimal presentations}

If $M\subseteq A^m$ is a stably free module given by the finite free resolution (\ref{equ6.2.11})
with $S=A$, then the Corollary \ref{6.3.4} gives a procedure for computing a minimal presentation of
$M$. In fact, if $m\geq 3$, then $f_m$ has a left inverse (if not, $\text{\rm pd}(M)=m$, but this
is impossible since $M$ is projective). Hence, we compute the new finite
presentation (\ref{equ6.3.5}) and we will repeat the procedure until we get a finite presentation
as in (\ref{equ6.3.6}), which is a minimal presentation of $M$.

\begin{example}
Let us consider again the ring $A=\sigma(\mathbb{Q})\langle x,y\rangle$, with $yx=-xy+1$.
Let $M$ be the left $A$-module given by presentation $A^2/Im(F_1)$, where
\begin{center}
$F_1=\begin{bmatrix}y^2&xy-1\\-xy&x^2\end{bmatrix}$.
\end{center}
Regarding the deglex order on $Mon(A)$, with $y\succ x$, and the TOP order over $Mon(A^2)$ with
$\textbf{\emph{e}}_2>\textbf{\emph{e}}_1$, we have that $Syz_A(F_1)$ is generated by $(x,y)$.
So, the following exact sequence is obtained:

\begin{equation*}
\begin{diagram}
\node{0}\arrow{e} \node{A} \arrow{e,t}{F_2} \node{A^2} \arrow{e,t}{F_1} \node{A^2} \arrow{e,t}{\pi} \node{M} \arrow{e} \node{0}
\end{diagram}
\end{equation*}

where $F_2:=\begin{bmatrix} x & y \end{bmatrix}^T$. Note that $F_2^{T}$ has a right inverse:
$G_2^T=\begin{bmatrix} y\\ x \end{bmatrix}$; thus, from Corollary \ref{6.3.4} we get
the following finite presentation for $M$:

\begin{equation}\label{MinimalPresentation}
\begin{diagram}
\node{0} \arrow{e} \node{A^2} \arrow{e,t}{h_1} \node{A^3} \arrow{e,t}{h_0} \node{M} \arrow{e} \node{0}
\end{diagram}
\end{equation}
with $H_1^T=\begin{bmatrix}F_1^T & G_2^T\end{bmatrix}$ and
$h_0=\begin{bmatrix} f_0 & 0\end{bmatrix}^T$. In the Example \ref{rightinverse2},
we showed that $H_1^T$ has a right inverse; moreover, one right
inverse for $H_1^T$ is
\begin{center}
$L_1^T=\begin{bmatrix}0&-1\\-1 &0 \\ x& y \end{bmatrix}$.
\end{center}
In consequence, (\ref{MinimalPresentation}) is a minimal presentation for $M$, and $M$ turns out to
be a stably free module.
\end{example}

\section{Computing free bases}\label{ComputingBases}
\noindent
In the \cite{Gallego3} and \cite{Quadrat}, it is presented a matrix constructive proof of a result due Stafford about stably free modules.
\begin{theorem}\label{7.3.6}
Let $S$ be a ring. Then any stably free $S$-module $M$ with ${\rm rank}(M)\geq \text{\rm sr}(S)$ is
free with dimension equals to ${\rm rank}(M)$.
\end{theorem}
\begin{proof}
See Theorem 1 in \cite{Gallego3}.
\end{proof}
In the proof of such affirmation,  the following fact is necessary.
\begin{proposition}\label{6.2.29}
Let $S$ be a ring and $\textbf{v}:=\begin{bmatrix}v_1 & \dots & v_r\end{bmatrix}^T$ an unimodular
stable column vector over $S$, then there exists $U\in E_r(S)$ such that
$U\textbf{v}=\textbf{e}_1$.
\end{proposition}
\begin{proof}
By completeness, we include the proof of this fact (see Proposition 38 in \cite{Quadrat}). There exist elements $a_1,\dots, a_{r-1}\in S$ such that
\begin{equation}\label{newunimodular}
\textbf{\emph{v}}':=(v_1',\dots, v_{r-1}')^T\in Um_c(r-1,S), \text{ with }
v_i':=v_i+a_iv_r, 1\leq i\leq r-1.
\end{equation}
Consider the matrix
\begin{equation}\label{matrixE1}
E_1:=\begin{bmatrix} 1 & 0 & 0 & \cdots & 0 & a_1\\
0 & 1 & 0 & \cdots & 0 & a_2\\
\vdots & \vdots & \vdots & \vdots & \vdots & \vdots \\
0 & 0 & 0 & \cdots & 1 & a_{r-1}\\
0 & 0 & 0 & \cdots & 0 & 1\\
\end{bmatrix}\in E_r(S);
\end{equation}
then $E_1v=(v_1',\dots, v_{r-1}',v_r)^T$. Since that
$\textbf{\emph{v}}':=(v_1',\dots, v_{r-1}')\in Um_c(r-1,S)$, there
exists $b_1,\dots,b_{r-1}\in S$ such that
$\sum_{i=1}^{r-1}b_iv_i'=1$, and hence,
$\sum_{i=1}^{r-1}(v_1'-1-v_r)b_iv_i'=v_1'-1-v_r$. Let
$v_i'':=(v_1'-1-v_r)b_i$, $1\leq i\leq r-1$ and
\begin{equation}\label{matrixE2}
E_2:=\begin{bmatrix} 1 & 0 & 0 & \cdots & 0 & 0\\
0 & 1 & 0 & \cdots & 0 & 0\\
\vdots & \vdots & \vdots & \vdots & \vdots & \vdots \\
0 & 0 & 0 & \cdots & 1 & 0\\
v_1'' & v_2'' & v_3'' & \cdots & v_{r-1}'' & 1\\
\end{bmatrix}\in E_r(S);
\end{equation}
then $E_2E_1v=(v_1',\dots,v_{r-1}', v_1'-1)^T$. Moreover, let
\begin{equation}\label{matrixE3}
E_3:=\begin{bmatrix} 1 & 0 & 0 & \cdots & 0 & -1\\
0 & 1 & 0 & \cdots & 0 & 0\\
\vdots & \vdots & \vdots & \vdots & \vdots & \vdots \\
0 & 0 & 0 & \cdots & 1 &0\\
0 & 0 & 0 & \cdots & 0 & 1\\
\end{bmatrix}\in E_r(S),
\end{equation}
then $E_3E_2E_1v=(1,v_2',\dots,v_{r-1}',v_1'-1)^T$. Finally, let
\begin{equation}\label{matrixE4}
E_4:=\begin{bmatrix} 1 & 0 & 0 & \cdots & 0 & 0\\
-v_2' & 1 & 0 & \cdots & 0 & 0\\
\vdots & \vdots & \vdots & \vdots & \vdots & \vdots \\
-v_{r-1}' & 0 & 0 & \cdots & 1 &0\\
-v_1'+1 & 0 & 0 & \cdots & 0 & 1\\
\end{bmatrix}\in E_r(S),
\end{equation}
then $E_4E_3E_2E_1v=\textbf{\emph{e}}_1$ and $U:=E_1E_2E_3E_4\in
E_r(S)$.
\end{proof}

For an effective calculation of a basis of $M$, we start establishing an algorithm  for to calculate the elementary
matrix $U$ in the Proposition \ref{6.2.29}:

\begin{center}
\fbox{\parbox[c]{12cm}{
\begin{center}
{\rm \textbf{Algorithm for computing $U$ in Proposition \ref{6.2.29} }}
\end{center}
\begin{description}
\item[]{\rm \textbf{INPUT}:} An unimodular stable column vector $\textbf{\emph{v}}=\begin{bmatrix} v_1 &\cdots& v_r \end{bmatrix}^T$ over $S$.

\item[]{\rm \textbf{OUTPUT}:} An elementary matrix $U\in M_r(S)$ such that $U\textbf{\emph{v}}=\textbf{\emph{e}}_1$.

\item[]{\rm \textbf{DO}:}
\begin{enumerate}
\item Compute $a_1,\ldots,a_{r-1}\in S$ such that (\ref{newunimodular}) holds.
\item Compute the matrix  $E_1$ given in (\ref{matrixE1}).
\item Calculate the elements $b_1,\ldots,b_{r-1}\in S$ with the property that
 $\sum_{i=1}^{r-1}b_iv_i'=1$, with $v_i'=v_i+a_iv_r$ for $1\leq i\leq r-1$.
\item Define $v_i'':=(v_i'-1-v_r)b_i$ for $1\leq i\leq r-1$, and compute the matrices $E_2$,
$E_3$ and $E_4$ given in (\ref{matrixE2})-(\ref{matrixE4}).
\end{enumerate}

\item[]{\rm \textbf{RETURN}:}  $U:=E_4E_3E_2E_1$.

\end{description}}}
\end{center}

We will illustrate below this algorithm.

\begin{example}\label{ExampleStable}
For this example, we consider the \textit{Quantum Weyl Algebra} $A_2(J_{a,b})$. Recall this $\Bbbk$-algebra is generated by the
variables $x_1,x_2,\partial_1,\partial_2$, with the relations (depending upon parameters
$a,b\in \Bbbk$):
\begin{align*}
x_1x_2=&x_2x_1+ax_1^2\\
\partial_2\partial_1=&\partial_1\partial_2+b\partial_2^2\\
\partial_1x_1=&1+x_1\partial_1+ax_1\partial_2\\
\partial_1x_2=&-ax_1\partial_1-abx_1\partial_2+x_2\partial_1+bx_2\partial_2\\
\partial_2x_1=&x_1\partial_2\\
\partial_2x_2=&1-bx_1\partial_2+x_2\partial_2.
\end{align*}
When $a=b=0$, we have that $A_2(J_{0,0})\cong A_2(\Bbbk)$ for any field $\Bbbk$ (see \cite{Fujita} for more properties).
It is not difficult to show that $A_2(J_{a,b})\cong \sigma(\Bbbk[x_1,\partial_2])\langle x_2,\partial_1 \rangle$.
Take $\Bbbk=\mathbb{Q}$, $a=0$ and $b=-1$. Thus, the relations in this
ring are given by:
\begin{align*}
x_1x_2=&x_2x_1\\
\partial_2\partial_1=&\partial_1\partial_2-\partial_2^2\\
\partial_1x_1=&1+x_1\partial_1\\
\partial_1x_2=&x_2\partial_1-x_2\partial_2\\
\partial_2x_1=&x_1\partial_2\\
\partial_2x_2=&1+x_1\partial_2+x_2\partial_2.
\end{align*}
$E_4(A_2(J_{0,-1}))$ it will denote the group generated by all elementary matrices of size $4\times 4
$ over $A_2(J_{0,-1})$. Let $\textbf{\emph{v}}=\begin{bmatrix}\partial_2+x_1
&\partial_2+\partial_1&x_2&\partial_1\end{bmatrix}^{T}$, then
$\textbf{\emph{u}}=\begin{bmatrix}\partial_1 &-\partial_2&0&-x_1\end{bmatrix}$ is such that
$\textbf{\emph{u}}\textbf{\emph{v}}=1$, whereby $\textbf{\emph{v}}\in Um_c(4,A_2(J_{0,-1}))$.
Moreover, the column vector $\textbf{\emph{v}}'=\begin{bmatrix}\partial_2+x_1
&\partial_2&x_2\end{bmatrix}^{T}$ has a left inverse $\textbf{\emph{u}}'=\begin{bmatrix}0
&x_2-x_1&\partial_2\end{bmatrix}$, so $\textbf{\emph{v}}$ is a stable unimodular column. In this
case, $a_1=0$, $a_2=-1$, $a_3=0$ and the matrix $E_1$ is given by
\begin{center}
$E_1=\begin{bmatrix}1&0&0&0\\ 0&1&0&-1\\ 0&0&1&0\\ 0&0&0&1\end{bmatrix}$.
\end{center}
With this elementary matrix we get
$E_1\textbf{\emph{v}}=\begin{bmatrix}\partial_2+x_1&\partial_2&x_2&\partial_1\end{bmatrix}^T$.
If we define $v_1'':=0$, $v_2'':=(\partial_2+x_1-1-\partial_1)(x_2-x_1)$, $v_3''=(\partial_2+x_1-1-\partial_1)\partial_2$
and
\begin{center}
$E_2=\begin{bmatrix}1&0&0&0\\ 0&1&0&-1\\ 0&0&1&0\\ 0&v_2''&v_3''&1\end{bmatrix}$,
\end{center}
we obtain $E_2E_1\textbf{\emph{v}}=\begin{bmatrix}\partial_2+x_1&\partial_2&x_2&\partial_2+x_1-1 \end{bmatrix}^T$.
Finally, if we define
\begin{center}
$E_3=\begin{bmatrix}1&0&0&-1\\0&1&0&0\\0&0&1&0\\0&0&0&1\end{bmatrix}\in E_4(A_2(J_{0,-1}))$,
$E_4=\begin{bmatrix}1&0&0&0\\-\partial_2&1&0&0\\-x_2&0&1&0\\-\partial_2-x_1+1&0&0&1\end{bmatrix}\in E_4(A_2(J_{0,-1}))$
\end{center}
and $U:=E_4E_3E_2E_1\in E_4(A_2(J_{0,-1}))$, then we have $U\textbf{\emph{v}}=\textbf{\emph{e}}_1$.
\end{example}

The proof of Theorem \ref{7.3.6} allows us to establish an algorithm to compute
a basis for $M$, when $M$ is a stably free module given by a minimal presentation
\begin{equation}\label{minimalstablybasis}
0\rightarrow S^s\xrightarrow{f_1}S^r\xrightarrow{f_0}M\rightarrow 0,
\end{equation}
with $g_1:S^r\to S^s$ such that $g_1\circ f_1=i_{S^s}$, and ${\rm rank}(M)=r-s\geq {\rm sr}(S)$.

\begin{center}
\fbox{\parbox[c]{12cm}{
\begin{center}
{\rm \textbf{Algorithm for computing bases}}
\end{center}
\begin{description}
\item[]{\rm \textbf{INPUT}:} $F_1=m(f_1)$ such that $F_1^T\in M_{s\times r}$ has a right inverse
$G_1^{T}\in M_{r\times s}$, and satisfies $r-s\geq {\rm sr}(S)$.

\item[]{\rm \textbf{OUTPUT}:} A matrix $U\in M_{r}(S)$ such that $UG_1^T=\begin{bmatrix}I_s & 0\end{bmatrix}^T$;
by Lemma \ref{matricialhermite2} the set $\{(U^{T})^{(s+1)},\ldots,(U^{T})^{(r)}\}$ is a basis for
$M$, where $(U^{T})^{(j)}$ denotes the $j$-th column of $U^T$ for $s+1\leq j\leq r$.

\item[]{\rm \textbf{INITIALIZATION}:} $i=1$, $V=I_r$.

{\rm \textbf{WHILE}} $i<r$ \textbf{DO}:

\begin{enumerate}
\item Denote by $\textbf{\emph{v}}_i\in S^{r-i+1}$ the column vector given by taking
the last $r-i+1$ entries of the $i$-th column of $VG_1^T$.
\item Apply the previous algorithm to compute $L_i\in E_{r-i+1}(S)$ such that
$L_i\textbf{\emph{v}}_i=\textbf{\emph{e}}_1$.
\item Define the matrix $U_i:=\begin{bmatrix}I_{i-1} & 0\\ 0 & L_i\end{bmatrix}\in E_{r}(S)$ for $i>1$, and $U_1:=L_1$.
\item $i=i+1$
\end{enumerate}

\quad {\rm \textbf{RETURN}}  $U:=PU_sV$, where $P$ is an adequate elementary matrix.

\end{description}}}
\end{center}

\begin{example}
Let $A$ be the \textit{Quantum Weyl Algebra} $A_2(J_{a,b})$ considered in Example
\ref{ExampleStable}, with $\Bbbk=\mathbb{Q}$, $a=0$ and $b=-1$. In order to illustrate the previous
algorithm, take $M=A^{6}/Im(F_1)$, where
\begin{center}
$F_1=\begin{bmatrix}0&\partial_1\\x_2&\partial_2\\0&-x_1\\ \partial_1&0\\ x_1 &1\\ \partial_2&-1 \end{bmatrix}$.
\end{center}
Using the algorithm described in Corollary \ref{RightInverseAlg}, the deglex order over $Mon(A)$,
with $x_2>\partial_1$, and the TOPREV order on $Mon(A^6)$, with
$\textbf{\emph{e}}_1>\textbf{\emph{e}}_2$, it is possible to show that $F_1^T$ has a right inverse
given by:
\begin{center}
$G_1^T=\begin{bmatrix}x_1\partial_1&x_1\\ 0&0\\ \partial_1^2& \partial_1\\x_1 & 0\\-\partial_1 &0\\0&0 \end{bmatrix}$.
\end{center}
Hence, we have the following minimal presentation for $M$:
\begin{equation}
0\rightarrow A^2\xrightarrow{F_1}A^6\xrightarrow{\pi}M\rightarrow 0,
\end{equation}
where $\pi$ is the canonical projection. Thus, $M$ is a stably free $A$-module with $\rm{rank}(M)=4$. Since
lKdim$(A)=3$ (see \cite{Fujita}, Theorem 2.2), then ${\rm sr}(A)\leq 4$ and by the Theorem \ref{7.3.6}, $M$
is free with dimension equals to ${\rm rank}(M)$. We will use the previous algorithm for computing a basis of $M$.\\
$\centerdot$ \textit{Step 1.} Let $V=I_6$ and $\textbf{\emph{v}}_1$ the first column of $VG_1^T$, i.e.,
\begin{center}
$\textbf{\emph{v}}_1=\begin{bmatrix}x_1\partial_1 &0&\partial_1^2&x_1&-\partial_1 &0 \end{bmatrix}^T$,
\end{center}
then $\textbf{\emph{v}}_1\in Um_{c}(6,A)$ and
$\textbf{\emph{u}}_1=\begin{bmatrix}0&x_2&0&\partial_1&x_1&-\partial_1\end{bmatrix}$ is such that
$\textbf{\emph{u}}_1\textbf{\emph{v}}_1=1$. Note that
$\textbf{\emph{v}}_1'=\begin{bmatrix}x_1\partial_1 &0&\partial_1^2&x_1&-\partial_1\end{bmatrix}^T$
is trivially unimodular. Applying to $\textbf{\emph{v}}_1$ the first algorithm of the current
section, we have that $E_1=I_6$,
\begin{center}
$E_2=\begin{bmatrix}1&0&0&0&0&0\\0&1&0&0&0&0\\0&0&1&0&0&0\\0&0&0&1&0&0\\0&0&0&0&1&0\\
0&(x_1\partial_1-1)x_2&0&(x_1\partial_1-1)\partial_1&(x_1\partial_1-1)x_1&1\end{bmatrix}$, \\
$E_3=\begin{bmatrix}1&0&0&0&0&-1\\0&1&0&0&0&0\\0&0&1&0&0&0\\0&0&0&1&0&0\\0&0&0&0&1&0\\
0&0&0&0&0&1\end{bmatrix}$ and,
$E_4=\begin{bmatrix}1&0&0&0&0&0\\0&1&0&0&0&0\\-\partial_1^2&0&1&0&0&0\\-x_1&0&0&1&0&0\\ \partial_1&0&0&0&1&0\\
-x_1\partial_1+1&0&0&0&0&1\end{bmatrix}$.
\end{center}
We can check that {\footnotesize
\begin{center}
$U_1:=E_4E_3E_2E_1=\begin{bmatrix}1&-(x_1\partial_1-1)x_2&0&-(x_1\partial_1-1)\partial_1&-(x_1\partial_1-1)x_1&-1\\
0&1&0&0&0&0\\
-\partial_1^2&\partial_1^2(x_1\partial_1-1)x_2&1&\partial_1^2(x_1\partial_1-1)\partial_1&
\partial_1^2(x_1\partial_1-1)x_1&\partial_1^2\\
-x_1&x_1(x_1\partial_1-1)x_2&0&x_1(x_1\partial_1-1)\partial_1+1&x_1(x_1\partial_1-1)x_1&x_1\\
\partial_1&-\partial_1(x_1\partial_1-1)x_2&0&-\partial_1(x_1\partial_1-1)\partial_1&-\partial_1(x_1\partial_1-1)x_1+1&-\partial_1\\
-x_1\partial_1+1&x_1\partial_1(x_1\partial_1-1)x_2&0&x_1\partial_1(x_1\partial_1-1)\partial_1&x_1\partial_1(x_1\partial_1-1)x_1
&x_1\partial_1\end{bmatrix}\in E_{6}(A)$
\end{center}
}
and
\begin{center}
$U_1G_1^{T}=\begin{bmatrix}1&x_1\\0&0\\0&-x_1\partial_1^2-\partial_1\\0&-x_1^2\\0&x_1\partial_1+1\\
0&-x_1^2\partial_1\end{bmatrix}$.
\end{center}

$\centerdot$ \textit{Step 2.} Make $V:=U_1$ and let $\textbf{\emph{v}}_2$ be the column vector
given by taking the last five entries of the $2$-th column of $VG_1^T$; i.e.,
$\textbf{\emph{v}}_2=\begin{bmatrix}0&-x_1\partial_1^2-\partial_1&-x_1^2&x_1\partial_1+1&-x_1^2\partial_1\end{bmatrix}^T$.
Note that $\textbf{\emph{u}}_2=\begin{bmatrix}0&-x_1&\partial_1^2&3&0\end{bmatrix}$ satisfies
$\textbf{\emph{u}}_2\textbf{\emph{v}}_2=1$, thus $v_2\in Um_{c}(5,A)$. Moreover,
$\textbf{\emph{v}}_2'=\linebreak\begin{bmatrix}0&-x_1\partial_1^2-\partial_1&-x_1^2&  x_1\partial_1+1\end{bmatrix}$
is unimodular with $\textbf{\emph{u}}_2'=\begin{bmatrix}0&-x_1&\partial_1^2&3\end{bmatrix}$ such
that $\textbf{\emph{u}}_2'\textbf{\emph{v}}_2'=1$, and hence $\textbf{\emph{v}}_2$ is stable. Using
the algorithm at the beginning of this section, we have that $E_1=I_5$,
\begin{center}
$E_2=\begin{bmatrix}1&0&0&0&0\\0&1&0&0&0\\0&0&1&0&0\\0&0&0&1&0\\
0&-(-1+x_1^2\partial_1)x_1&(-1+x_1^2\partial_1)\partial_1^2&3(-1+x_1^2\partial_1)&1\end{bmatrix}$,
$E_3=\begin{bmatrix}1&0&0&0&-1\\0&1&0&0&0\\0&0&1&0&0\\0&0&0&1&0\\0&0&0&0&1\end{bmatrix}$ and,
$E_4=\begin{bmatrix}1&0&0&0&0\\x_1\partial_1^2+\partial_1&1&0&0&0\\x_1^2&0&1&0&0\\-x_1\partial_1-1&0&0&1&0\\ 1&0&0&0&1\end{bmatrix}$.
\end{center}
Making the respective calculations, we have that {\tiny
\begin{center}
$L_2:=E_4E_3E_2E_1=\begin{bmatrix}1&(-1+x_1^2\partial_1)x_1&-(-1+x_1^2\partial_1)\partial_1^2&-3(-1+x_1^2\partial_1)&-1\\
x_1\partial_1^2+\partial_1&1+(x_1\partial_1^2+\partial_1)(-1+x_1^2\partial_1)x_1&
-(x_1\partial_1^2+\partial_1)(-1+x_1^2\partial_1)\partial_1^2&-3(x_1\partial_1^2+\partial_1)(-1+x_1^2\partial_1)&
-(x_1\partial_1^2+\partial_1)\\
x_1^2&x_1^2(-1+x_1^2\partial_1)x_1&1-x_1^2(-1+x_1^2\partial_1)\partial_1^2&-3x_1^2(-1+x_1^2\partial_1)&-x_1^2\\
-(x_1\partial_1+1)&-(x_1\partial_1+1)(-1+x_1^2\partial_1)x_1&(x_1\partial_1+1)(-1+x_1^2\partial_1)\partial_1^2&
1+3(x_1\partial_1+1)(-1+x_1^2\partial_1)&x_1\partial_1+1\\1&0&0&0&0\end{bmatrix}$.
\end{center}
} and $L_2\textbf{\emph{v}}_2=\textbf{\emph{e}}_1\in A^{5}$. Define
$U_2:=\begin{bmatrix}1&0\\0&L_2\end{bmatrix}$; then
\begin{center}
$U_2U_1G_1^T=\begin{bmatrix}1&x_1\\0&1\\0&0\\0&0\\0&0\\ 0&0\end{bmatrix}$.
\end{center}
Finally, if
\begin{center}
$P_1:=\begin{bmatrix}1&-x_1&0&0&0&0\\0&1&0&0&0&0\\0&0&1&0&0&0\\0&0&0&1&0&0\\0&0&0&0&1&0\\
0&0&0&0&0&1\end{bmatrix}$, then $UG_1^T=\begin{bmatrix}1&0\\0&1\\0&0\\0&0\\0&0\\ 0&0\end{bmatrix}$,
\end{center}
where $U:=P_1U_2U_1$. Thus, a basis for $M$ is given by
$\{\pi(U_{(3)}),\pi(U_{(4)}),\pi(U_{(5)}),\pi(U_{(6)})\}$,
with $U_{(i)}^T$ denoting the transpose of $i$-th row of the matrix $U$, for $i=3,4,5,6$; i.e.,
{\small
\begin{align*}
U_{(3)}^T=&\begin{bmatrix}-x_1^3\partial_1^2+x_1\partial_1^3-4x_1^2\partial_1-2x_1\\
(x_1\partial_1^2+\partial_1)(1-x_1\partial_1^2x_2+x_1^3\partial_1^3x_2+\partial_1x_2)\\
1+(x_1\partial_1^2+\partial_1)(-1+x_1^2\partial_1)x_1\\
(x_1\partial_1^2+\partial_1)(x_1^3\partial_1^4-x_1\partial_1^3+2\partial_1^2-x_1\partial_1^3)\\
(x_1\partial_1^2+\partial_1)(\partial_1x_1-x_1\partial_1^2x_1+x_1^3\partial_1^3x_1-3x_1^2\partial_1+3)\\
(x_1\partial_1^2+\partial_1)(-\partial_1+x_1^2\partial_1^2-x_1\partial_1)+\partial_1^2\end{bmatrix},\\
U_{(4)}^T=&\begin{bmatrix}x_1^2\partial_1-x_1^4\partial_1^2+x_1^3\partial_1-x_1^2-x_1\\
x_1^2+(-x_1^2\partial_1+x_1^4\partial_1^2-x_1^3\partial_1+x_1)(x_1\partial_1-1)x_2\\
-x_1^3+x_1^5\partial_1+x_1^4\\
-x_1^3\partial_1^3+x_1^5\partial_1^4+2x_1^2\partial_1^2-x_1\partial_1-x_1^4\partial_1^3+1\\
-x_1^4\partial_1^2-x_1^3\partial_1+x_1^6\partial_1^3+3x_1^5\partial_1^2-3x_1^4\partial_1+3x_1^2\\
-x_1^2\partial_1+x_1^4\partial_1^2-x_1^3\partial_1+x_1\end{bmatrix},\\
U_{(5)}^T=&\begin{bmatrix}-x_1\partial_1^2+x_1^3\partial_1^3+2x_1^2\partial_1^2-x_1\partial_1+1\\
x_1\partial_1(-1+x_1\partial_1^2x_2-x_1^3\partial_1^3x_2)-x_1^3\partial_1^3x_2-1\\
-(x_1\partial_1+1)(-1+x_1^2\partial_1)x_1\\
(x_1\partial_1+1)(x_1\partial_1^3-x_1^3\partial_1^4+x_1^2\partial_1^3-\partial_1^2)\\
(x_1\partial_1+1)(x_1\partial_1^2x_1-x_1^3\partial_1^3+3x_1^2\partial_1-3)-x_1^2\partial_1^2+2x_1\partial_1+1\\
-(x_1\partial_1+1)(-\partial_1+x_1^2\partial_1^2-x_1\partial_1)-\partial_1\end{bmatrix},\,
U_{(6)}^T=\begin{bmatrix}0\\ 1\\ 0\\ 0 \\0\\ 0\end{bmatrix}
\end{align*}}

\end{example}



\begin{thebibliography}{200}

\bibitem{Chyzak}\textbf{Chyzak, F., Quadrat, A. and Robertz, D.}, \textit{Effective algorithms for parametrizing linear
control systems over Ore algebras}, Appl. Algebra Engrg. Comm. Comput., 16, 2005, 319-376.

\bibitem{Quadrat4}\textbf{Cluzeau, T. and Quadrat, A.}, \textit{Factoring and decomposing a class of linear functional systems}, Lin. Alg. And Its Appl., 428, 2008, 324-381.

\bibitem{Cohn1}\textbf{Cohn, P.}, \textit{Free Ideal Rings and Localizations in General Rings}, Cambridge University Press,
    2006.

\bibitem{Fujita}\textbf{Fujita, H.}, \textit{Global and Krull Dimensions of Quantum Weyl Algebras}, Journal of Algebra,
216, 1999, 405-416.

\bibitem{Gallego6}\textbf{Gallego, C.}, \emph{Matrix methods for projective modules over $\sigma$- PBW extensions},
Ph.D. thesis, Universidad Nacional de Colombia, Bogotá, 2015.

\bibitem{Gallego2}\textbf{Gallego, C. and Lezama, O.}, \textit{Gröbner bases for ideals of skew $PBW$ extensions},
Communications in Algebra, 39, 2011, 50-75.

\bibitem{Gallego3}\textbf{Gallego, C. and Lezama, O.}, \textit{Matrix approach to noncommutative
stably free modules and Hermite rings},  Algebra and Discrete Mathematics, 18 (1), 2014, 110-139.

\bibitem{Gallego4}\textbf{Gallego, C. and Lezama, O.}, \textit{d-Hermite rings and skew PBW extensions}, São Paulo Journal of Mathematical Sciences, pp 1-13, First online: 25 August 2015.

\bibitem{Gallego5}\textbf{Gallego, C. and Lezama, O.}, \textit{Projective modules and Gröbner bases
for skew PBW extensions}, to appear in ``Algebraic and Symbolic Computation
Methods in Dynamical Systems'' in the Springer series ``Advances in Delays and Dynamics''.

\bibitem{Gago}\textbf{Gago-Vargas, J.}, \textit{Bases for projective modules in $A_n(k)$}, J. Symb. Comp., 36, 2003, 845-853.

\bibitem{Lam}\textbf{Lam, T.Y.}, \textit{Serre's Problem on Projective Modules
}, Springer Monographs in Mathematics, Springer, 2006.

\bibitem{Lam1}\textbf{Lam, T.Y.}, \textit{Lectures on Modules and Rings}, GTM 189, Springer, 1999.

\bibitem{Lezama5}\textbf{Lezama, O.}, \textit{Matrix and Gröbner Methods in Homological Algebra over Commutative Polynomial Rings},
Lambert Academic Publishing, 2011.

\bibitem{lezamareyes1}\textbf{Lezama, O. and Reyes, M.}, {\em Some homological properties of skew $PBW$
extensions}, Comm. in Algebra, 42, (2014), 1200-1230.

\bibitem{MacDonald}\textbf{MacDonald, B.}, \textit{Linear Algebra over Commutative Rings}, Marcel Dekker, 1984.

\bibitem{Quadrat}\textbf{Quadrat, A., Robertz, D.}, \textit{Computation of bases of free modules
over the Weyl algebras}, J. Symb. Comp., 42, 2007, 1113-1141.

\bibitem{Stafford2}\textbf{Stafford, J.T.}, \textit{Module
structure of Weyl algebras}, J. London Math. Soc. 18, 1978, 429-442.

\end{thebibliography}
\end{document}